\documentclass[11pt]{article}

 \usepackage{amsthm}

\usepackage{graphics} 
\usepackage{epsfig} 
\usepackage{mathptmx} 
\usepackage{amsmath} 
\usepackage{amssymb}  
\newcommand{\N}{{\mathbf{N}}}
\newcommand{\C}{{\mathbf{C}}}
\newcommand{\R}{{\mathbf{R}}}
\newtheorem{theorem}{Theorem}[section]
\newtheorem{rem}{Remark}[section]
\newtheorem{prop}[theorem]{Proposition}

\title{ Simultaneous approximate tracking of density matrices for a system of Schr\"{o}dinger equations
}

\author{Thomas Chambrion
\thanks{This work was supported by BQR R\'egion Lorraine-Nancy Universit\'e}
\thanks{The author is with IECN/INRIA/Nancy Universit\'e, UMR7502, BP 239
54506 Vand{\oe}uvre, France} \\
{{\tt\small Thomas.Chambrion@esstin.uhp-nancy.fr}}}

\begin{document}

\maketitle
\thispagestyle{empty}
\pagestyle{empty}

\begin{abstract}
We consider a non-resonant system of finitely many bilinear Schr\"odinger equations with discrete spectrum driven by the same scalar control. We prove that this system can approximately track any given system of trajectories of density matrices, up to the phase of the coordinates. The result is valid both for bounded and unbounded Schr\"odinger operators. The method used relies on finite-dimensional control techniques applied to Lie groups. We provide also an example showing that no approximate tracking of both modulus and phase is possible 
\end{abstract}

\section{Introduction}
\subsection{Physical Context}

The Schr\"odinger equation describes the evolution of 
the probability distribution of the position of a particle in the space. 
The evolution of the Schr\"{o}dinger equation  can be 
modified by acting on the electric field, e.g., through 
the action of a laser.

We will be interested in this paper in
non-relativistic and non-stochastic Schr\"{o}dinger equations
on a domain (i.e., an open connected subset) $\Omega$ of $\R^d$
that is either
bounded 
or equal to the whole $\R^d$ ($d\in\N$). To each equation, we associate
a {\it Schr\"odinger operator}  defined as
$$(x\mapsto \psi(x))\mapsto(x\mapsto -\Delta \psi (x)+ V(x)\psi(x)),\ \ \ \ x\in\Omega,$$
 where $\psi$ denotes the wave function
  and the real-valued function $V$ is called the \emph{potential} of the Schr\"odinger operator. The wave function verifies
$\int_\Omega \psi^2=1$.
We assume moreover that $V$ is extended as $+\infty$ on $\R^d\smallsetminus \Omega$, so that, in the case $\Omega$ bounded, $\psi$
satisfies the boundary condition $\psi|_{\partial \Omega}=0$.
The controlled Schr\"odinger equation with one scalar control
is the evolution equation
\begin{equation} \label{EQ-main}
i \frac{d \psi}{dt}= -\Delta \psi(x,t) +V(x)\psi(x,t)+ u(t) W(x)\psi(x,t),
\end{equation}
where the real-valued function $W$ is the \emph{controlled potential}.

The control function $u:[0,T]\rightarrow \R$ is chosen among a set of admissible control functions in order to steer the quantum particle from its initial state to a prescribed target. A classical result of Ball, Marsden and Slemrod asserts that in general, exact controllability is hopeless (see \cite{BMS} and \cite{Turinici} for precise statements and a proof).

The approximate controllability of one particular system of the type (\ref{EQ-main}) has already be proved by Beauchard using Coron's return method  (see \cite{Beauchard} and references therein). Approximate controllability results for general systems under generic hypotheses were proved later with completely different methods in \cite{Schrod}.

An interesting question is the simultaneous control of several Schr\"odinger equations of type (\ref{EQ-main}), describing the evolution of several particles, driven by only one control.
The most elementary case is the one where identical particles are submitted to the same control field (that is, Schr\"odinger operator and controlled potential are the same for every particle), but the initial positions and the targets differs from one particle to another. This problem is often referred as \emph{control of density matrices} by physicists. This case has already been successfully addressed in  \cite{Schrod}. 

In this work, we study the approximate controllability of a system of density matrices, covering among others the following two physicals situations:

 In the first situation, we consider a set of particles of different natures, that is, their Schr\"odinger operators are all different. Initial states and targets can be arbitrary. The control function $u$ is the same for all particles. 
 
 In the second situation, we consider a set of identical particles, with the same Schr\"odinger operator, but with different controlled potentials. This second situation is of particular interest since it is usually easier to select a set of identical particles than to ensure that they are all submitted to the same excitation. The control function $u$ is the same for all particles.

Our main result is that, under non-resonnance and connectedness conditions, in the both cases above, whatever the sources and the target are, it is possible to chose an arbitrary small piecewise constant control function $u$ in such a way that every particle approximately reaches its goal (with an arbitrary precision), and moreover that given any path joining the sources and the target, it can be followed (up to the phase) by the particles with an arbitrary precision. Our approach extends to the control of the density matrices of the considered systems.

\subsection{Mathematical framework}

We give below the abstract mathematical framework 
which will be used to formulate and prove 
the controllability results  later applied to 
the Schr\"{o}dinger equation.
The fact 
that the Schr\"{o}dinger equation fits the abstract framework has already been discussed in \cite[Section 3]{Schrod}.

Let $U$ be a subset of $\R$. Let $p$ be a nonzero positive integer. For $1\leq i \leq p$, let $H_i$ be an Hilbert space,
$n_i$ be an integer, $A_i:D(A_i)\subset H_i\rightarrow H_i$ be a densely defined (not necessarily bounded) essentially self-adjoint operator  and for $1\leq j \leq n_i$, let $B_{i,j}:D(B_{i,j})\subset H_i\rightarrow H_i$ be a densely defined (not necessarily bounded) linear operator. 

  We assume  that $\left ( (A_i)_{1\leq i \leq p},(B_{i,j})_{1\leq i \leq p, 1\leq j \leq n_i},U \right )$ satisfies
the following three conditions: (H1) $A_i$ and $B_{i,j}$ are skew-adjoint, 
(H2) there exists an orthonormal basis  $(\phi^k_i)_{k=1}^\infty$ of $H_i$ made of
eigenvectors of $A_i$, and all these eigenvectors are associated to simple eigenvalues
(H3) $\phi^k_i\in D(B_{i,j})$ for $1\leq i \leq p$, $1 \leq j\leq n_i$. A crucial consequence of these hypotheses is that
 for every $1\leq i\leq p$, $1\leq j \leq n_i$ and every $u\in U$, $A_i +u B_{i,j}$ has a  self adjoint extension on a dense subdomain of $H_i$ and  generates 
a group of unitary transformations
$e^{t(A_i+u B_{i,j})}:H_i\to H_i$.
In particular, $e^{t(A_i+u B_{i,j})}(S_i)=S_i$ for every $u\in U$, $1\leq i \leq p$, $1 \leq j \leq n_i$ and every $t\geq 0$, where
$S_i$ is the unit sphere of $H_i$.

 We consider the $\sum_{i=1}^p n_i$ conservative diagonal single input control systems
\begin{equation} \label{EQ_main_system}
\left \{
\begin{array}{lcl} 
\frac{d \psi_{1,1}}{dt}(t)&=&A_1(\psi_{1,1}(t)) +u(t) B_{1,1}(\psi_{1,1}(t))\\
\frac{d \psi_{1,2}}{dt}(t)&=&A_1(\psi_{1,2}(t)) +u(t) B_{1,2}(\psi_{1,2}(t))\\
&\vdots&\\
\frac{d \psi_{1,n_1}}{dt}(t)&=&A_1(\psi_{1,n_1}(t)) +u(t) B_{1,n_1}(\psi_{1,n_1}(t))\\
\frac{d \psi_{2,1}}{dt}(t)&=&A_2(\psi_{2,1}(t)) +u(t) B_{2,1}(\psi_{2,1}(t))\\
&\vdots&\\
\frac{d \psi_{p,n_p}}{dt}(t)&=&A_p(\psi_{p,n_p}(t)) +u(t) B_{p,n_p}(\psi_{p,n_p}(t))\\
\end{array}
\right.
\end{equation} 
with initial conditions to be specified later.

A point $\psi^0=(\psi_{1,1}^0,\psi_{1,2}^0,\ldots,\psi_{1,n_1}^0,\psi_{2,1}^0,\ldots,\psi_{p,n_p}^0)$ of $H=\prod_{i=1}^p \left ( H_i \right )^{n_i}$ and a piecewise constant function $u:[0,T]\to U$, $u=\sum_{l=1}^{L} \chi_{[t_l,t_{l+1})} u_l$ being given,
we say that the solution of (\ref{EQ_main_system}) with initial condition $\psi^0 \in  H$
 and corresponding to the  control function $u:[0,T]\to U$ is the 
  curve $$
  \begin{array}{llcl} \psi:& [0,T] & \rightarrow & H \\ &t&\mapsto& \psi(t)=(\psi_{1,1}(t),\psi_{1,2}(t),\ldots,\psi_{p,n_p}(t))\end{array}$$
 defined by   
\begin{eqnarray} \label{solu}
\psi_{i,j}(t)=e^{(t-\sum_{l=1}^{l_t-1} t_l)(A_i +u_l B_{i,j})}
&\circ e^{t_{l_t-1}(A_i +u_{l_t-1} B_{i,j})}&\\\circ \cdots &\circ e^{t_1(A_i+u_1  B_{i,j})}(\psi_{i,j}^0),
\end{eqnarray}
$1\leq i \leq  p$, $1 \leq j \leq n_i$ by
where $\sum_{l=1}^{l_t-1} t_l\leq t<\sum_{l=1}^{l_t} t_l$ and 
$u(\tau)=u_j$ if $\sum_{l=1}^{j-1} t_l\leq \tau<\sum_{l=1}^{j} t_l$.
Notice that such a $\psi_{i,j}(\cdot)$  satisfies, for every $n\in \N$ and almost every $t\in [0,T]$ 
the differential equation
\begin{equation}\label{very_weak}
\frac d{d t}\langle \psi_{i,j}(t),\phi^n_i\rangle=\langle \psi_{i,j}(t),(A_i+u(t)B_{i,j})\phi^n_i\rangle.
\end{equation}

A piecewise constant function $u:[0,T]\rightarrow \R$ being given, the propagator of the 
control system in $H_i$ $$ \frac{d \psi_{i,j}}{dt}(t)=A_i(\psi_{i,j}(t)) +u(t) B_{i,j}(\psi_{i,j}(t))$$ will be denoted by 
$\Phi_{i,j}$. By definition, $$\Phi_{i,j}(t,\psi_{i,j}^0)=\psi_{i,j}(t)=e^{(t-\sum_{l=1}^{k-1} t_l)(A_i +u_k B_{i,j})}\circ e^{t_{k-1}(A_i +u_{k-1} B_{i,j})}\circ \cdots \circ e^{t_1(A_i+u_1  B_{i,j})}(\psi_{i,j}^0)$$ for any $t$ in $[0,T]$ and any $\psi_{i,j}^0$ in $H_i$.

 The propagator of the whole system (\ref{EQ_main_system}) is denoted by $\Phi=\prod_{i=1}^p \prod_{j=1}^{n_i}\Phi_{i,j}$.

 For $1\leq i \leq p$, $1\leq j\leq n_i$, and $(k,l)\in \N^2$, we define also the numbers $a_i(k,l)=\langle A_i \phi^i_k,\phi^i_l\rangle$ and $b_{i,j}(k,l)=\langle B_{i,j} \phi^i_k,\phi^i_l\rangle$.

 A finite sequence $(k_1,k_2,\ldots,k_l)$ of $\N$ is said to \emph{connect} the two levels $k$ and $k'$ for the diagonal conservative diagonal single-input control system $ ( A_i,B_{i,j})$ if $k_1=k$, $k_l=k'$ and  $\prod_{q=1}^{l-1} b_{i,j}(k_q,k_{q+1}) \neq 0$.
 
 A subset $S$ of $\N^2$ is called a \emph{connectedness chain} of $(A_i,B_{i,j})$ if for arbitrary large $N$, for $(k,k')$ in $\N^2$, $1\leq k,k'\leq N$, there exists a finite sequence $(q_1,q_2, \ldots,q_l)$ in $[0,N]$ that connects $k$ and $k'$ for $(A_i,B_{i,j})$ and such that $(q_{r},q_{r+1})$ belongs to $S$ for every $1 \leq r \leq l-1$.

\subsection{Main result}

\begin{theorem}\label{TheSuiviPropagateur}
Assume that the system (\ref{EQ_main_system}) satisfies (i) the concatenation of the spectrum of $A_1$, $A_2$, $\ldots$, $A_p$ is a $\mathbf{Q}$-linearly independent family (ii) for any $1 \leq i_0 \leq p$, for any $1 \leq j_0 \leq n_{i_0}$, there exists a connectedness chain $S$ of $(A_{i_0},B_{i_0,j_0})$ such that
for any $s$ in $S$, for any $j$ in $[1,n_{i_0}]\setminus \{ j_0 \}$, $|b_{i_0,j_0}(s)|\neq |b_{i_0,j}(s)|$   and (iii) $U$ contains a neigborhood of zero.

Let $c:[0,T]\rightarrow \prod_{i=1}^p{L(H_i,H_i)}^{n_i}$ be a continuous curve such that $c(0)=Id_H$ and $c=\prod_{i=1}^p \prod_{j=1}^{n_i} c_{i,j}$ is a curve taking value in the set of the Hermitian operators of $H$ such that for every $1\leq i \leq p$, $1\leq j \leq n_i$, for every time $t$, $c_{i,j}(t):H_i\rightarrow H_i$ is unitary in $H_i$. Let $N$ be an integer. Then, for every $\epsilon>0$,  there exist $T_u>T$ and a piecewise constant control $u:[0,T_u]\rightarrow U $,  such that the corresponding propagator $\Phi:[0,T_u] \times H \rightarrow H$  of system (\ref{EQ_main_system})  satisfies i) for every $t$ in $[0,T_u]$, there exists $s$ in $[0,T]$ such that $\big| |\langle \Phi_{i,j}(s,\phi_l^i), \phi_k^i \rangle | -|\langle c_{i,j}(t)(\phi_l^i),\phi_k^i \rangle | \big|<\epsilon$ for every $k$ in $\mathbf{N}$, $1 \leq i \leq p$, $1 \leq j \leq n_i$, $1 \leq l \leq N$, and 
ii) $\| \Phi_{i,j}(T_u,\phi_l^i) - c_{i,j}(T)(\phi_l^i) \|<\epsilon$ for every $1\leq i \leq p$, $1\leq j \leq n_i$, $1 \leq l \leq N$.   
\end{theorem}

\subsection{Content of the paper}
 To prove Theorem \ref{TheSuiviPropagateur}, we will use a method introduced for the control of the Navier-Stokes equation in \cite{Navier} and adapted to the Schr\"{o}dinger equation in \cite{Schrod}. 
 
 In Section \ref{SECGalerkyn}, we explain how to choose a Galerkyn approximation of the original infinite dimensional control problem (\ref{EQ_main_system}) in some space $SU(k)$. In Section \ref{SECTrackingGalerkyn}, we use the Lie group structure of $SU(k)$ to  compute the dimensions of some Lie subalgebras of $\mathfrak{su}(k)$ and to prove that the Galerkyn approximations obtained in Section \ref{SECGalerkyn} have some good tracking properties.  The proof of Theorem \ref{TheSuiviPropagateur} and an estimation of the $L^1$ norm of the control are given in Section \ref{SECProofInfiDim}, where we prove that the original system (\ref{EQ_main_system}) share the tracking properties of the Galerkyn approximations established in Section \ref{SECTrackingGalerkyn}. A partial counterpart of Theorem \ref{TheSuiviPropagateur} (impossibility of approximate tracking of both the modulus and the phase) is given in Section \ref{SECNoexactTracking}.

\section{Choice of Galerkyn approximations} \label{SECGalerkyn}

\subsection{Control and time-reparametrization}

We may assume without loss of generality that $U$ has the special form $U=]0,\delta]$. Remark that, if $u\neq 0$, $e^{t(A_i+u B_{i,j})}=e^{tu((1/u)A_i+ B_{i,j})}$. 
Associate with any piecewise constant function $u=\sum_{l=1}^{k-1} \chi_{[t_l,t_{l+1})} u_l$ in $PC([0,T_u], U)$ the function $v=$ $\sum_{l=1}^{k-1} \chi_{[\tau_l,\tau_{l+1}[} 1/u_l$ $\in PC([0,T_v], 1/U)$, with $T_v=\sum_l u_l (t_{l+1}-t_l)$ and $\tau_l$ defined by induction by $\tau_1=t_1$ and $\tau_{l+1}=\tau_l+u_l(t_{l+1}-t_l)$. Up to the time and control reparametrization given above, it is enough to prove Theorem \ref{TheSuiviPropagateur} for the  system $((B_{i,j})_{i=1..p,j=1..n_i},$ $ (A_i)_{i=1..p},$ $ [\frac{1}{\delta},+\infty[)$:
\begin{equation} \label{EQ_main_system_reparam}
\left \{ \begin{array}{lcl}
\frac{d \psi_{1,1}}{dt}(t)&=&v(t) A_1(\psi_{1,1}(t)) + B_{1,1}(\psi_{1,1}(t))\\
\frac{d \psi_{1,2}}{dt}(t)&=&v(t) A_1(\psi_{1,2}(t)) + B_{1,2}(\psi_{1,2}(t))\\
&\vdots&\\
\frac{d \psi_{p,n_p}}{dt}(t)&=&v(t) A_p(\psi_{p,n_p}(t)) + B_{p,n_p}(\psi_{p,n_p}(t))\\
\end{array}
\right.
\end{equation}
where the set of admissible controls is the set $PC\left ( \mathbf{R}^+,\left \lbrack\frac{1}{\delta},+\infty \right \lbrack  \right )$. 
\begin{rem}\label{rem-equiv-temps-L1norm}
 A feature of this reparametrization of the control from $u:[0,T_u]\rightarrow U$ to $v:[0,T_v]\rightarrow 1/U$ is that $\|u\|_{L^1}=T_v$. 
\end{rem}

\subsection{Galerkyn approximation}

For a fixed piecewise constant control $v:\R^+ \rightarrow 1/U$ and a fixed family $(\psi_{i,j}^0)_{i=1..p,j=1..n_i}$ in $H=\prod_{i=1}^p H_i^{n_i}$, we consider
the solution $(\psi_{i,j})_{i=1..p,j=1..n_i}:t\rightarrow H$ of the system (\ref{EQ_main_system_reparam}) of conservative diagonal single-input control systems with initial conditions $(\psi_{i,j}^0)_{1\leq i\leq p, 1 \leq j \leq n_i}$. 

For $1\leq i \leq p$, $1\leq j \leq n_i$, $k\in \N$ , we define the function $x^{k}_{i,j}=\langle \psi_{i,j},\phi^k_{i}\rangle:\R \rightarrow \C$.
With our definition of solution, $x^k_{i,j}$ is absolutely continuous and for almost all $t$ in $\R^+$
$$\frac{d}{dt}x^k_{i,j}= v(t) a_i(k,k) + \sum_{l \in \N} b_{i,j}(k,l) x^l_{i,j}. $$

\begin{prop}\label{PR0_Tracking_lemma}
For every $1 \leq i \leq p$ and any continuous curve $s:\lbrack 0, T_s \rbrack \rightarrow H_i$ taking value in the unit sphere of $H_i$ (that is, $\|s(t)\|=1$ for all $t$
in $\lbrack 0, T_s \rbrack$),  define the family $f_l=|\langle s, \phi_i^l \rangle|^2$, $l \in \N$. Then, for any strictly positive
 $\epsilon$, there exists an integer $N(\epsilon)$ such that  for all $t$
in $\lbrack 0, T_s \rbrack$, $\sum_{l=1}^{N(\epsilon)}f_l(t)>1-\epsilon$.
\end{prop} 

\begin{proof} 
Define $g_l=\sum_{k=1}^l f_k$. For all integer $l$, $g_l:[0,T_s]\rightarrow \R$ is a continuous function. 
For any $\epsilon$, the set ${\cal O}_l(\epsilon)=\{ t \in \lbrack 0, T_c \rbrack / g_l(t) >1-\epsilon \}$
is an open subset of $\lbrack 0, T_c \rbrack$ (for the topology induced on $\lbrack 0, T_c \rbrack$ by the
Euclidean distance in $\mathbf{R}$). Since for all $t$, $g_l(t)$ is a non-decreasing function of $l$, it is
also clear that for any $\epsilon$, ${\cal O}_l(\epsilon)\subset {\cal O}_{l+1}(\epsilon)$. 

Since, for all $t$, $g_l(t)$ tends to $1$ as $l$ tends to infinity, we can say that for all $\epsilon$,
$$ \cup_{l \in \mathbf{N}} {\cal O}_l(\epsilon) = \lbrack 0, T_c \rbrack.$$
Using the compactness of $\lbrack 0, T_c \rbrack$, we can find a finite set $l_1,l_2,\ldots l_k$ of integer
such that $$ \cup_{j=1}^{k} {\cal O}_{l_j}(\epsilon) = \lbrack 0, T_c \rbrack.$$
Define now $N(\epsilon)=\mbox{sup}(l_1,l_2,\ldots,l_p)$. For all $t$ in $\lbrack 0, T_c \rbrack$,
$\sum_{k=1}^{N(\epsilon)}f_k(t)>1-\epsilon$.
\end{proof}

We define $\pi_{i}^m:H_i\rightarrow \mathbf{C}^m$ by $\pi_i^m(v)=\sum_{k=1}^m \langle v, \phi_i^k\rangle e_k^m$ for every $v$ in $H_i$, where $e_k$ is the $k^{th}$ element of the canonical basis of $\mathbf{C}^m$.

\begin{prop}\label{PRO_lemme_exist_Galerkyn}
Fix a reference curve $c:[0,T]\rightarrow \prod_{i=1}^p L(H_i,H_i)^{n_i}$ as in the hypotheses of Theorem \ref{TheSuiviPropagateur}, $\epsilon>0$ and $N$ a positive integer.  Then, for every $(i,j)$, there exists a continuous curve $M_{i,j}:[0,T]\rightarrow SU(m)$ such that 
$\|\pi_{i}^m(c_{i,j}(t)\phi_i^k) -M_{i,j}(t) \pi_{i}^m \phi_i^k\| <\epsilon$ for every $t$ in $[0,T]$ and every $k$ in $\{1..N\}$.
\end{prop}

\begin{proof}
For $1\leq i \leq p$, $1 \leq j \leq n_i$, $1\leq k \leq N$, apply the Proposition \ref{PR0_Tracking_lemma}  to the  curve $s:t\mapsto c_{i,j}(t)(\phi_{k}^{i})$ to get some integer $m_{i,j,k}$. For any integer $m\geq \sup \{m_{i,j,k}, 1\leq i \leq p, 1 \leq j \leq n_i, 1\leq k \leq N\}$, for every $t$ in $[0,T_c]$, for every $1\leq i \leq p$, $1\leq j \leq n_i$  and every $1 \leq k \leq N$, $\|\Pi^m_i\left ((c_{i,j}(t)(\psi_{i}^k)\right ) -c_{i,j}(t)(\psi_{i}^k) \|_i<\epsilon$.  

The application $F_{i,j}^m:[0,T]\rightarrow \left(\C^{m}\right)^N$ defined for every $t$ in $[0,T]$ by $F_{i,j}^m(t)=(\pi_{i,j}^m(c_{i,j}\phi_i^k))_{k=1..N}$ associate with every $t$ a family of vectors of $\left(\C^{m}\right)^N$ that is almost orthonormal. Actually, denote ${\cal S}: \left(\C^{m}\right)^N \rightarrow \left(\C^{m}\right)^N $ the map that associates to any linearly independent family $(v_1,..,v_N)$ of $N$ vectors of $\left(\C^{m}\right)^N$ the orthonormal family ${\cal S} (v_1,..,v_N) $ deduced from $(v_1,..,v_N)$ by the Schmidt orthonormalization procedure. For any $i \in \{1..p\}$, $j \in \{1..n_i\}$, $t\in [0,T]$, $F_{i,j}^m(t) -{\cal S}F_{i,j}^m(t)$ tends to $0 \in \left(\C^{m}\right)^N$  as $m$ tends to infinity. Noticing that $E_{i,j}^m:t\mapsto \|F_{i,j}^m(t) -{\cal S}F_{i,j}^m(t)\|$ is continuous, considering the open set ${\cal Q}_m(\epsilon)=\{t\in [0,T]|E_{i,j}^m(t)<\epsilon\}$ and applying the same compacity argument as in the proof of Proposition \ref{PR0_Tracking_lemma}, we find an integer $m_{i,j}$ such that $\|F_{i,j}^{m_{i,j}}(t) -{\cal S}F_{i,j}^{m_{i,j}}(t)\|<\epsilon$ for every $t\in [0,T]$. Choosing $m=\sup_{1\leq i \leq p, 1 \leq j \leq n_i} m_{i,j}$ we have $\|(\pi_{i}^m(c_{i,j}(t)\phi_i^k))_{k=1..N}-{\cal S}F_{i,j}^{m}(t)\|<\epsilon$ for every $t$ in $[0,T]$, every $i$ in $\{1..p\}$, every $j$ in $\{1..n_i\}$. To conclude, we choose $M_{i,j}(t)$ in such a way that its first $N$ columns coincide with the first $N$ elements of ${\cal S}F_{i,j}^{m_{i,j}}(t)$.
\end{proof}

For $1\leq i \leq p$, $1 \leq j \leq n_i$, define the $r \times r$ matrices $A_i^{r}=\lbrack a_{i}(k,l) \rbrack_{1\leq k,l \leq r}$ and $B_{i,j}^{r}=\lbrack b_{i,j}(k,l) \rbrack_{1\leq k,l \leq r}$.

The Galerkyn approximation of order $m$ of the system (\ref{EQ_main_system_reparam} is 
\begin{equation}\label{EQ_main_system_Galerkyn_Cn}
\left \{\begin{array}{lcl}
\frac{d x_{1,1}}{dt}(t)&=&v(t) A^m_1(x_{1,1}(t)) + B^m_{1,1}(x_{1,1}(t))\\
\frac{d x_{1,2}}{dt}(t)&=&v(t)A^m_1(x_{1,2}(t)) + B^m_{1,2}(x_{1,2}(t))\\
&\vdots&\\
\frac{d x_{p,n_p}}{dt}(t)&=&v(t)A^m_p(x_{p,n_p}(t)) + B^m_{p,n_p}(x_{p,n_p}(t))
\end{array}
\right.
\end{equation} 
The system (\ref{EQ_main_system_Galerkyn_Cn}) defines a control system on the differentiable manifold
$\prod_{i=1}^p \prod_{j=1}^{n_i} \mathbf{C}^m$. Since the system (\ref{EQ_main_system_Galerkyn_Cn}) is linear, it is possible to lift it to the group of matrices of the resolvent,
\begin{equation}\label{EQ_main_system_Galerkyn}
\left \{\begin{array}{lcl}
\frac{d x_{1,1}}{dt}(t)&=&v(t)A^m_1(x_{1,1}(t)) + B^m_{1,1}(x_{1,1}(t))\\
\frac{d x_{1,2}}{dt}(t)&=&v(t)A^m_1(x_{1,2}(t)) +u B^m_{1,2}(x_{1,2}(t))\\
&\vdots&\\
\frac{d x_{p,n_p}}{dt}(t)&=&v(t)A^m_p(x_{p,n_p}(t)) + B^m_{p,n_p}(x_{p,n_p}(t))
\end{array}
\right.
\end{equation} 
 which a system in $\prod_{i=1}^p \prod_{j=1}^{n_i} SU(m)$, with matrix unknowns $(x_{1,1},x_{1,2},..,x_{p,n_p})$.

We now proceed to a technical change of variable (variation of the constant) and define for every piecewise constant control function $v$ in $PC(\R,1/U)$, every positive $t$   and  every $1\leq i\leq p$, $1\leq j \leq n_i$, $y_{i,j}(t)=e^{-A^m_{i} \int_{0}^t v(s)ds} x_{i,j}(t)$.

Recalling that for all $m\times m$ matrices $a,b$, $e^{-a}be^a=e^{ad (a)}b$, one checks that $(y_{i,j})_{1\leq i\leq p, 1\leq j \leq n_i}$ verifies
\begin{equation}\label{EQ_main_system_Galerkyn_phase}
\left \{\begin{array}{lcl}
\frac{d y_{1,1}}{dt}(t)&=&e^{ad (\int_0^tv(s)ds A^m_{1})}  B^m_{1,1} y_{1,1}(t)  \\
\frac{d y_{1,2}}{dt}(t)&=&e^{ad (\int_0^tv(s)ds A^m_{1})}  B^m_{1,2} y_{1,2}(t)\\
&\vdots&\\
\frac{d y_{p,n_p}}{dt}(t)&=&e^{ad (\int_0^tv(s)ds A^m_{p})} B^m_{p,n_p} y_{p,n_p}(t) 
\end{array}
\right.
\end{equation} 
The system (\ref{EQ_main_system_Galerkyn_phase}) defines a control system on the differentiable manifold $\prod_{i=1}^p SU(m)^{n_i}$, and for every positive $t$, every $1\leq i\leq p$, every $1\leq j \leq n_i$ and every $1\leq k \leq m$, $|\langle x_{i,j}(t),\phi_{i}^k \rangle |=|\langle y_{i,j}(t),\phi_{i}^k \rangle |$.

Using the canonical injection of $\prod_{i=1}^p SU(m)^{n_i}$ in $SU(m\sum_{i=1}^p n_i)$, we can write the system (\ref{EQ_main_system_Galerkyn_phase}) as 
\begin{equation}\label{EQ_main_system_Galerkyn_phase_mp}
\frac{d y}{dt}={\cal F}_u(t)y,
\end{equation} 
where ${\cal F}_u:\mathbf{R} \rightarrow \mathfrak{su}\left (m \sum_{i=1}^p n_i \right)$ is a map that associates to any $t$ the $\left ( m \sum_{i=1}^p n_i \right ) \times \left ( m \sum_{i=1}^p n_i \right )$ diagonal block matrix constructed from $\left ( \sum_{i=1}^p n_i \right )\times \left ( \sum_{i=1}^p n_i \right ) $ blocks of size $m\times m$, and whose diagonal is $(e^{ad A_{i}^{m} \int_0^t v} B^{(m)}_{i,j})_{1\leq i \leq p,1\leq j\leq n_i}$.

\section{Tracking properties of the Galerkyn approximations} \label{SECTrackingGalerkyn}

\subsection{Finite dimensional tracking in semi-simple Lie groups}

First, we have to recall some classical definitions and results for invariant Lie groups control systems. The now classical control extensions techniques used here have been introduced by Kupka in the 70's, see \cite{Review_Yuri} and references therein for details. Most of the material below can be found with proofs and references in \cite{agrachev_chambrion}.

Let $G$ be a semi-simple compact Lie groups, with Lie algebra $\mathfrak{g}=T_{Id}G$ and Lie bracket $[,]$. The Killing form (see \cite[Chapter III]{Helgason} for details) is negative definite on the Lie algebra $\mathfrak{g}$ of $G$. Its opposite $\langle,\rangle$ is a scalar product on $\mathfrak{g}$ that can be extended on the whole tangent bundle $TG$ by the left action of $G$ over itself. The smooth manifold $G$ endowed with this bi-invariant scalar product $\langle, \rangle$ turns into a Riemannian manifold,  whose distance is denoted by $d_G$.

 Consider a smooth  right invariant control system on $G$ of the form
$$
\left \{ \begin{array}{lcl} \frac{d}{dt}g(t)&=&dR_{g(t)}f(u(t))\\
                             g(0) & = & g_0
         \end{array}
	 \right.  ~~~~~~~~~~~~~~~~~~~~(\Sigma)
$$
where $U$ is a subset of $\R$,  $u:\R \rightarrow U$ is a control function to be chosen is one of the following class of regularity ${\cal K}=$ $\{$ absolutely continuous, measurable bounded, locally integrable, piecewise constant $\}$, $f:U\rightarrow \mathfrak{g}$ is a smooth application,   $g_0$ is a given initial condition and $dR_ab$ denotes the value of the differential of the right translation by $a$ taken at point $b$. If $G=SU(n)$, $[,]$ is the standard matrix commutator, the exponential map $\exp:\mathfrak{g}\rightarrow G$ is the standard matrices exponential,  the elements $a$ of $G$ are the matrices with determinant equal to one and such that $^t\bar a a=Id$ and the elements $b$ of $\mathfrak{g}$ are the zero trace matrices such that $^t \bar{b}+b=0$, and $dR_ab= b\times a$ where $\times$ is the standard matrices multiplication.

We define the set ${\cal V}=\overline{\mbox{conv}\{ f(u),u\in U\}}$ as the topological closure of the convex hull of all admissible velocities at point $Id$.

It is obvious that the topological closure of the convex hull of all admissible velocities at point $g$ is $dR_g({\cal V})$.

We will need the following pretty standard relaxation result:
\begin{prop}\label{PRO_tracking_semi_simple}
Let $P$ be a Lie-subgroup of $G$ with Lie algebra $\mathfrak{p}$.  
If $\cal V$ contains some bounded symmetric set $S$ such that $\mathfrak{p} \subset Lie (S)$, then for any continuous curve $c:[0,T]\rightarrow P$, for any $\epsilon>0$, for any regularity class $k$ of $\cal K$, there exist $T_u>0$, a control function $u:[0,T_u]\rightarrow U$ of class $k$, and an increasing continuous bijection $\phi:[0,T_u]\rightarrow [0,T]$
such that the trajectory $g:[0,T_u]\rightarrow G$ of $(\Sigma)$ with control $u$ and initial condition $c(0)$ satisfies
(i) $d_G(c(\phi(t)),g(t))<\epsilon$ for every $t$ in $[0,T_u]$ (ii) $\phi(T_u)=T$ and (iii) $c(T)=g(T_u)$. 
\end{prop}

\begin{proof}
Fix $\epsilon>0$ and a continuous curve $c:[0,T]\rightarrow P$. Up to translation by $c(0)$, one may  assume  without loss of generality that $c(0)=Id_G$. By a classical density argument, one may also assume that for all $t$, $c(t)=\exp(v(t))$ where $v:[0,T]\rightarrow \mathfrak{p}$ is a piecewise constant function. It is hence enough to study the case for which $c(t)=\exp(tv)$ for any $t$ with some constant $v$ in $\mathfrak{p}$. 

Define by induction $S^{(0)}={S}$ and $S^{(i+1)}=span S+[S,S^{i}]$ for any integer $i$. By hypothesis, $S^{\infty}=\cup_{i\in \N} S^{(i)} \supset\mathfrak{p}$. We proceed by induction on $i$ to prove the result if $c$ has the special form $c(t)=\exp(tv)$ for any $t$ with some constant $v$ in $S^{(i)}$. 

The case where $v$ is in $S$ follows from \cite[Theorem 8.2]{book2}.
 
 Fix $i$ in $\N$ and assume that the result is known for any element in $S^{(i)}$. Choose $v$ in $S^{(i+1)}$, and
 write $v$ as a the limit of a (fixed, this is a consequence of the theorem of Caratheodory) linear combination of brackets of elements of $S$ and $S^{(i)}$:
 $$v=\lim_{n\rightarrow \infty} \sum_{i \in I} \lambda_i [{v'}_i^n,{v''}_i^n] $$ where $I$ is a finite set, $\lambda_i$ is a (constant) real number, $({v'}_i^n)_n$ is a sequence of elements of $S$, converging to some $v'_i$ and $({v''}_i^n)_n$ is a sequence of elements of $S^{(i)}$, converging to some $v''_i$. Using once again the symmetry of $S$ and a time reparametrization, one may assume that $0<\lambda_i\leq 1$ for every $i$ and $\sum_{i \in I} \lambda_i=1$. 
 
 Recall now
 the Baker-Campbell-Hausdorff formula (see \cite{Helgason}): for any $u$, $v$ in $\mathfrak{g}$, 
 $$\exp(-tu) \exp(-tv) \exp(tu) \exp(tv)= \exp\left (t^2 [u,v] + o_{0}(t^2) \right ),$$
that is
 $d_G(\exp(-tu) \exp(-tv) \exp(tu) \exp(tv),\exp\left (t^2 [u,v] \right )=t^2 \alpha(t)$, for some function $\alpha:\R\rightarrow \R$  with limit zero at zero. For $\tau_0$ small enough to be fixed later, define the $4\tau_0$-periodic piecewise constant function $F:\R\rightarrow \mathfrak{g}$ by
$F(t)=v$ for $0\leq t \leq \tau_0$, $F(t)=u$ for $\tau_0< t \leq 2\tau_0$, $F(t)=-v$ for $2 \tau_0< t \leq 3\tau_0$ and
$F(t)=-u$ for $3 \tau_0< t \leq 4\tau_0$, and consider the curve $g:t\mapsto \exp(F(t))$.
For any $n\in \N$, for any $t$ in $[0,n \tau_0]$,  $d_G(g(t),\exp\left (t^2 [u,v] \right )< \sum_{k=1}^n \tau_0^2 \alpha(\tau_0)=n\tau_0^2 \alpha(\tau_0) $. 

Fix $\eta>0$  and $T>0$. Choose $\tau_0$ small enough such that $\alpha(\tau_0)<\frac{\eta}{T}$ and $n=\frac{T}{\tau_0}$. One gets
 $d_G(g(t),\exp\left (t^2 [u,v] \right )<\eta $ for any $t$ such that $0\leq t \leq T$.
Apply this last inequality with $u$ in $S$ and $v$ in $S^{i}$. The proof of Proposition \ref{PRO_tracking_semi_simple} follows from \cite[Theorem 8.2]{book2}.
\end{proof}

To obtain trackabillity properties for the system (\ref{EQ_main_system_Galerkyn_phase}), it is enough to check that the finite dimensional systems (\ref{EQ_main_system_Galerkyn_phase})
satisfies the conditions on $S$ given in Proposition \ref{PRO_tracking_semi_simple} for a suitable $\mathfrak{p}$.

We define the set ${\cal V}=\overline{\mbox{conv}(\{{\cal F}_v(t), v \in PC\left (\mathbf{R}, [\frac{1}{\delta},+\infty[ \right ), t\in \mathbf{R}^+\})}$.  In the sequel of this Section, we prove that under the hypotheses of Theorem \ref{TheSuiviPropagateur}, it is possible to find a set $S$ in  $\cal V$ satisfying the hypotheses of Proposition \ref{PRO_tracking_semi_simple}. 

\subsection{Some Lie algebraic methods}

 Fix an integer $m$ in $\mathbf{N}$, and, for $1 \leq i \leq p, 1 \leq j \leq n_i, 1\leq k,l\leq m$, denote  with $E_{i,j,k,l}$ the square matrix of order $m \left ( \sum_{q=1}^p n_q \right )$  whose entries are all zero but the one with index $(m(i-1)+m(j-1)+i,m(i-1)+m(j-1)+j)$ which is equal to one. (We consider $E_{i,j,k,l}$ as a block-matrix. The two first indices $(i,j)$ stands for the $m\times m$ block, the two last indices $(j,k)$ stand for the position of the non-zero entry inside the $m\times m$ block of index $(i,j)$.) 

\begin{prop}\label{PRO_lemme_convex_hull}
 Fix $r$ in $\N$. If $a$ and $b$ are two matrices of $\mathfrak(su(r))$ such that
$a$ is diagonal with $\mathbf{Q}$-linearly independent spectrum and $b$ has entries $b_{k,l}$, for $1\leq k,l\leq m$, then all the matrices $b_{k,l} E_{k,l} + b_{l,k} E_{l,k}$ belong to the set $\overline{\mbox{conv}(\{  Ad_{\exp(\int_0^t va)}b ; v\in PC \left (\mathbf{R}, [\frac{1}{\delta},+\infty[\right ), t\in \mathbf{R}^+\})}$. 
\end{prop}
\begin{prop}\label{PRO_lemme_convex_hull_rotation}
 Fix $r$ in $\N$. If $a$ and $b$ are two matrices of $\mathfrak(su(r))$ such that
$a$ is diagonal with $\mathbf{Q}$-linearly independent spectrum and $b$ has entries $b_{j,k}$, for $1\leq j,k\leq m$, then for every $\theta$ in $\mathbf{R}$, all the matrices $e^{i \theta} b_{j,k} E_{j,k} + e^{-i \theta} b_{k,j} E_{k,j}$ belong to the set $\overline{\mbox{conv}(\{ {e^{ad\int_0^t va}} b ; v\in PC \left (\mathbf{R},[\frac{1}{\delta},+\infty[ \right ), t\in \mathbf{R}^+\})}$. 
\end{prop}

\begin{proof}
The proof can be found in \cite[Appendix A]{agrachev_chambrion}.
\end{proof}

Applying Proposition \ref{PRO_lemme_convex_hull} to the set ${\cal V}$ defined in Section 3.1, one gets that all the matrices $\sum_{i=1}^p \sum_{j=1}^{n_i} b_{i,j}(k,l) E_{i,j,k,l}$ $+$ $b_{i,j}(l,k) E_{i,j,l,k}$ belong to ${\cal V}$. We define $S$ as the set of matrices $S=\{ \pm \sum_{i=1}^p \sum_{j=1}^{n_i} b_{i,j}(k,l)^i E_{i,j,k,l} + b_{i,j}(l,k) E_{i,j,l,k},  1\leq k,l\leq m\}$. Proposition \ref{PRO_lemme_convex_hull_rotation} (applied with $\theta=\pi$) proves that $S$ is actually in $\cal V$.
By definition, $S$ is symmetric and bounded. What remains to prove now is that the Lie algebra generated by $S$ is equal to $\mathfrak{p}=\prod_{i=1}^p \mathfrak{su}(m)^{n_i}$.

\subsection{Reduction to $SU(m)^{n_i}$}
\begin{prop}\label{PRO_calcul_AlgLie}
Choose any $1<leq i_0 \leq p$, $1 \leq j_0 \leq n_{i_0}$, $1\leq k_0,l_0 \leq m$ such that $(k_0,l_0)$ is in connectedness chain of $(A_{i_0},B_{i_0,j_0})$ and $b_{i_0,j}(k_0,l_0)\neq b_{i_0,j_0}(k_0,l_0)$ for every $j\neq j_0$. Then the matrix $b_{i_0,j_0}(k_0,l_0) E_{i_0,j_0,k_0,l_0} + b_{i_0,j_0}(l_0,k_0) E_{i_0,j_0,l_0,k_0}$ is in $Lie(S)$.
\end{prop}
\begin{proof} 
Since the matrices $A_i$ have $\mathbf{Q}$ linearly independent spectrum, it is enough to apply Proposition \ref{PRO_lemme_convex_hull} to see that  for any $1\leq i_0 \leq p$, $1\leq k,l \leq m$, every matrix   $\sum_{j=1}^{n_{i_0}} b_{i_0,j}(k_0,l_0)^i E_{i_0,j,k_0,l_0} + b_{i_0,j}(l_0,k_0) E_{i_0,j,l_0,k_0}$ is actually contained in $\cal V$, hence in $S$. If $n_{i_0}=1$, the Proposition \ref{PRO_calcul_AlgLie} is proved.

If $n_{i_0}>1$, the Proposition \ref{PRO_lemme_convex_hull} is not enough to guaranty that the matrix $b_{i_0,j_0}(k_0,l_0)^i E_{i_0,j_0,k_0,l_0} + b_{i_0,j_0}(l_0,k_0) E_{i_0,j_0,l_0,k_0}$ is actually contained in $\cal V$. Nevertheless, one can prove that these matrices are contained in the Lie algebra generated by $S$.

Indeed, by Proposition \ref{PRO_lemme_convex_hull}, $a=\sum_{j=1}^{n_{i_0}}  b_{i_0,j}(k_0,l_0) E_{i_0,j,k_0,l_0}+ b_{i_0,j}(l_0,k_0) E_{i_0,j,k_0,l_0}$ belongs to ${\cal V}$.  Using Proposition \ref{PRO_lemme_convex_hull_rotation} with $\theta=\frac{\pi}{2}$, one gets that
$b=\sum_{j=1}^{n_{i_0}} i b_{i_0,j}(k_0,l_0) E_{i_0,j,k_0,l_0}-i b_{i_0,j}(l_0,k_0) E_{i_0,j,k_0,l_0}$. Compute $[[a,b],b]=-4 $ $ \sum_{j=1}^{n_{i_0}} $ $ |b_{i_0,j}(k_0,l_0)|^2$ $\left ( b_{i_0,j}(k_0,l_0) E_{i_0,j,k_0,l_0}+ b_{i_0,j}(l_0,k_0) E_{i_0,j,k_0,l_0} \right )$, and by induction $ ad^k_{[a,b]}b=(-1)^k 2^{k+1} $ $\sum_{j=1}^{n_{i_0}}  |b_{i_0,j}(k_0,l_0)|^{2k}$ $\left ( b_{i_0,j}(k_0,l_0) E_{i_0,j,k_0,l_0}+ b_{i_0,j}(l_0,k_0) E_{i_0,j,k_0,l_0} \right )  $ for every  $k \in \N$ (see \cite{agrachev_chambrion} for details). A classical Vandermonde argument on the linear independence of the vectors $(|b_{i_0,j}(k_0,l_0)|^{r})_{1\leq j\leq n_{i_0}}$ gives the result. 
\end{proof}

The fact that $Lie(S)=\mathfrak{p}$  follows from Proposition \ref{PRO_calcul_AlgLie} by the hypothesis of connectedness (see \cite[Proposition 4.1]{Schrod} for a detailed computation).

\section{Infinite dimensional tracking}\label{SECProofInfiDim}

\subsection{Tracking in the phase variables}

For the proof of  Theorem \ref{TheSuiviPropagateur}, we follow the method introduced in \cite{Schrod}. From the application $c:\R \rightarrow L(H,H)$ and the tolerance $\epsilon$ given in the hypotheses of Theorem \ref{TheSuiviPropagateur}, we use the results presented in Section II.B to find an integer $m$, the finite dimensional control system (\ref{EQ_main_system_Galerkyn_phase}) and the trajectory $t\mapsto \prod_{i,j} M_{i,j}(t)$ to be tracked in $SU \left (m \sum_{i=1}^p n_i \right )$. Proposition \ref{PRO_tracking_semi_simple} gives the existence of some time $T_v>0$ and some control function $v$ in $PC([0,T_v],1/U)$ such that the corresponding trajectory $(y_{1,1},..,y_{p,n_p})$ of (\ref{EQ_main_system_Galerkyn_phase}) tracks the trajectory $t\mapsto \prod_{i,j} M_{i,j}(t) $ with an error less than $\epsilon$ on each coordinate. 

Since for every $1 \leq i \leq p$, $1\leq j \leq n_i$, $1\leq k \leq m$, the sequence  $\left ( b_{i,j}(k,l) \right )_{l \geq 1}$ is in $\ell^2$, there exists some $N_1$ in $\N$ such that $\sum_{l=N_1 +1}^{\infty} |b_{i,j}(k,l)|^2<\frac{\epsilon}{N T_v}$ for every $1 \leq i \leq p$, $1\leq j \leq n_i$, $1\leq k \leq m$. The next result asserts that any trajectory of the system (\ref{EQ_main_system_Galerkyn_phase_mp}) can actually be tracked (up to $\epsilon$), with the $N_1$-Galerkyn approximation of system  (\ref{EQ_main_system_reparam}). 
\begin{prop}\label{PRO_track_galerkyn_m_N1}
There exists a sequence $(v_k)_k$ in $PC\left ( \R^+,\left \rbrack \frac{1}{\delta},+\infty \right \lbrack \right )$ such that for every $1\leq i \leq p$, $1\leq j \leq n_i$, the sequence of matrix valued curves $t\mapsto e^{ad \int v A_i^{(m)}}B_{i,j}^{(m)}$ converges in the integral sense to  $t \mapsto \left ( \begin{array}{c|c} \prod_{i,j} M_{i,j}(t) & 0_{m,N_1-m} \\ \hline  0_{N_1-m,m} & G(t) \end{array} \right )$, where $t\mapsto G(t)$ is some continuous curve in $U(N_1-m)$.
\end{prop}
\begin{proof} The proof is a direct application of \cite[Claim 4.3]{Schrod}, dealing with the convergence of the sequence $e^{ad \int v_k \prod_{i} A_i^{(N_1)}}$. \end{proof}  

\begin{prop}\label{PRO_approx_phase}
For $k$ large enough, the control function $v=v_k$ given by Proposition \ref{PRO_track_galerkyn_m_N1} satisfies the conclusion (i) of Theorem \ref{TheSuiviPropagateur}.
\end{prop}
\begin{proof} This is a direct application of \cite[Claim 4.4]{Schrod}. \end{proof}

\subsection{Final phase adjustment}
After time reparametrization, we get a control function $u \in PC([0,T_u],U)$ from $v$.
Up to prolongation with the constant zero function, the control function $u:[0,T_u]\rightarrow U$ obtained in  Proposition \ref{PRO_approx_phase} can always be assumed to satisfy $T_u>T$ (the prolongation obviously  
still satisfies conclusion (i) of Theorem \ref{TheSuiviPropagateur}). 

To achieve the proof of Theorem \ref{TheSuiviPropagateur}, one has to change $u$ in such a way that it satisfies  the conclusion (ii) of Theorem \ref{TheSuiviPropagateur}. One gets the result with a straightforward application of \cite[Proposition 4.5]{Schrod}.

\subsection{Estimates of the $L^1$-norm of the control}

Combining the Remark \ref{rem-equiv-temps-L1norm} and the estimates of \cite[Prop 2.7-2.8]{agrachev_chambrion}, one gets an easily computable estimation of the $L^1$-norm of the control $u$. We denote with $\mu_{i,j}(t)=\sqrt{\langle M_{i,j}^{-1}(t)M_{i,j}'(t),M_{i,j}^{-1}(t)M_{i,j}'(t)\rangle}$ the velocity at time $t$ of the trajectory to be tracked in $SU(m)$. 
\begin{prop}
 In Theorem \ref{TheSuiviPropagateur}, one can choose the control $u$ in such a way that $$\| u \|_{L^1} \leq \frac{ \left ( \sum_{i=1}^p n_i \right )^{\frac{3}{2}}  N_1^2 \sum_{i,j} \|\mu_{i,j}\|_{L^1} }{\min_{i,j,0\leq k,l\leq N_1} |b_{i,j}(k,l)|}.$$
\end{prop}

This estimate is valid for every $(b_{i,j})_{i,j}$, yet is sometimes trivial or too conservative when some $b_{i,j}(j,k)$ is close to zero. For these anisotropic situations, when some directions are much easier to follow than others, one can obtain sharper estimates using  \cite[Theorem 2.13]{agrachev_chambrion}, the expressions being  slightly more intricate.

\section{To track both the phase and the modulus is impossible}\label{SECNoexactTracking}

In this Section, we give a partial counterpart to Theorem \ref{TheSuiviPropagateur}. Indeed, we exhibit an example for which it is not possible to track both the phase and the modulus. The proof can easily be extended to a wide range of systems.

Consider one single control system in an Hilbert space $H$
\begin{equation} \label{EQ_systeme_non_trackable}
\left \{ \begin{array}{lcl} \dot{x}&=& Ax + uBx\\ x(0)&=&\phi_1 \end{array} \right. 
\end{equation}
where $A:H\rightarrow H$ is a diagonal operator in the Hilbert base $(\phi_l)_{l\in \mathbf{N}}$ of $H$, with purely imaginary eigenvalues $\left ( i\lambda_l  \right )_{l \in \mathbf{N}}$ and $B$ is a skew adjoint operator whose domain contains $\phi_l$ for every $l$ in $\mathbf{N}$, satisfying 
$b_{i,j}=\langle B \phi_i,\phi_j \rangle \in \R$ for every $i,j$ in $\N$.  
Define as admissible control functions all piecewise constant functions $u:\mathbf{R} \rightarrow \R^+$.
For $l \in \mathbf{N}$, we note $x_l=\langle x,\phi_l \rangle$ the component of the solution of system (\ref{EQ_systeme_non_trackable}) and we define $a_l=\Re(x_l)$, $b_l=\Im(x_l)$.

\begin{rem}
 In the case where $B$ is bounded, it is possible to define solutions of (\ref{EQ_systeme_non_trackable}) for $u$ in $L^1(\R, \R^+)$.  The result and the proof below are easily extended to integrable controls that are not necessary piecewise constant (in particular, to controls that may be not essentially bounded).
\end{rem}

\begin{prop}\label{PRO_no_tracking}
  Assume $\lambda_1, b_{2,1}>0$. Then, for $\epsilon<\frac{b_{2,1}}{b_{2,1}+ \| B \phi_2 \|}$,  for every piecewise constant control function $u:\mathbf{R}\rightarrow \R^+$, there exists $\tau>0$, there exists $i$ in $\N$, $i>1$ such that $|x_i(\tau)|>\epsilon$. 
   \end{prop}
   In other words, it is not possible to track with an arbitrary precision the constant trajectory $x_1 \equiv 1$.
   
   \begin{proof} We proceed by contradiction and assume that there exists some admissible control function $u:\mathbf{R}\rightarrow \R^+$ such that the corresponding trajectory of (\ref{EQ_systeme_non_trackable}) remains $\epsilon$-close to $\phi_1$ for every time. 
   From system (\ref{EQ_systeme_non_trackable}), we see that
   $$\frac{d}{dt}x_1=i \lambda _1 x_ 1 + u \left ( \sum_{j=2}^{+\infty} \langle B \phi_1,\phi_j \rangle x_j \right ), $$
   that is
   \begin{eqnarray}
   \dot{a}_1 & = & -\lambda _1 b_1 + u \Re \left ( \sum_{j=2}^{+\infty} \langle B \phi_1,\phi_j \rangle x_j \right ), \label{EQ_a_1}\\
   \dot{b}_1 & = & \lambda_ 1 a_1 +  u \Im \left ( \sum_{j=2}^{+\infty} \langle B \phi_1,\phi_j \rangle x_j \right ). \label{EQ_b_1}
   \end{eqnarray}
For any positive $t$, the integration of (\ref{EQ_b_1}) on $[0,t]$ yields $b_1(t)=\lambda_1 \int_0^t a_1(s) ds +$ $ \int_0^t u(s) \sum_{i=2}^{\infty} b_{1,i}b_i(s)ds$, that is $$-\epsilon \| B \phi_1 \| \int_0^t u(s) ds <\int_0^t u(s) \sum_{i=2}^{\infty} b_{1,i}b_i(s)ds = b_1(t) -\lambda_1 \int_0^t a_1< \epsilon -\lambda_1 (1-\epsilon) t$$
and
\begin{equation}\label{EQ_minor_int_u}
 \int_0^t u(s)ds > \frac{\lambda_1 (1-\epsilon) t}{\epsilon \|B \phi_1\|}.
\end{equation}
Integrating now $\dot{a}_2(s)=-\lambda_2 b_2(s) + u(s) \sum_{i\neq 2} b_{2,i} a_i(s)$ on $[0,t]$ for any $t>0$, one finds $$a_2(t)\geq -\lambda_2 \epsilon -\int_0^t u(s) ds \|B \phi_2\| \epsilon + b_{2,1} \int_0^tu(s) a_1(s)ds \geq \left ( b_{2,1} (1-\epsilon) -\epsilon \|B \phi_ 2\| \right ) \int_0^t u(s) ds.$$
For $\epsilon$ small enough, $K=\left ( b_{2,1} (1-\epsilon) -\epsilon \|B \phi_ 2\| \right )>0$, and from (\ref{EQ_minor_int_u}), we get $a_2(t)\geq K t$ for every positive $t$. Hence,
$a_2(t)$ tends to infinity as $t$ tends to infinity, what is impossible since $|a_2|\leq \|x\|$ which is constant equal to $1$. This gives the desired contradiction.
   \end{proof}

\section{Acknowledgments}

The author is grateful to Andrei Agrachev that inspired this work, and to Mario Sigalotti for many corrections and suggestions.


\begin{thebibliography}{99}


\bibitem{Navier}Agrachev, A,  Kuksin, S, Sarychev, A, and Shirikyan, A, {\it
On finite-dimensional projections of distributions for solutions of randomly forced 2D Navier-Stokes equations.}
Ann. Inst. H. Poincar\'e Probab. Statist. 43 (2007), no. 4, 399--415. 

\bibitem{book2} Agrachev, A and Sachkov, Y {\it
Control theory from the geometric viewpoint.} Encyclopaedia of Mathematical Sciences, 87.
Control Theory and Optimization, II. Springer-Verlag, Berlin, 2004. xiv+412 pp. ISBN: 3-540-21019-9 

\bibitem{agrachev_chambrion}Agrachev, A and Chambrion, T
{\it An estimation of the controllability time for single-input systems on compact Lie groups.}
ESAIM Control Optim. Calc. Var. 12 (2006), no. 3, 409--441 (electronic).
J.G.F. Francis, The QR Transformation I, {\it Comput. J.}, vol. 4, 1961, pp 265-271.

\bibitem{BMS}Ball, J M, Marsden, J E and Slemrod, M {\it Controllability for distributed bilinear systems.}, SIAM J. Control Optim. 20 (1982), no. 4, 575--597. 

\bibitem{Beauchard} Beauchard, K, {\it Local controllability of a 1D Schr\"odinger equation}, J. Math. Pures et Appl., (2005), 84, 851--956.

\bibitem{Schrod}Chambrion, T, Mason, P, Sigalotti, M and Boscain, U {\it Controllability of the discrete-spectrum Schr\"odinger equation driven by an external field}, Annales de l'IHP, analyse non lin\'eaire, doi:10.1016/j.anihpc.2008.05.001.

\bibitem{Helgason} Helgason, S, {\it
Differential geometry and symmetric spaces.}
Pure and Applied Mathematics, Vol. XII. Academic Press, New York-London 1962 xiv+486 pp. 



\bibitem{Review_Yuri} Sachkov, Y, {\it Controllability of invariant systems on Lie groups and homogeneous spaces.} Dynamical systems, 8.
J. Math. Sci. (New York) 100 (2000), no. 4, 2355--2427.

\bibitem{Turinici} Turinici, G, {\it On the controllability of bilinear quantum systems},
  {Mathematical models and methods for ab initio Quantum Chemistry}, Lecture Notes in Chemistry 74,  Springer, 2000.

\end{thebibliography}
\end{document}